\documentclass[conference]{elsarticle}

\usepackage{xspace}
\usepackage{amsmath,amssymb,amsfonts}
\usepackage{algorithm}
\usepackage{algpseudocode}

\usepackage{amsthm}
\usepackage{graphicx}
\usepackage{url}
\usepackage{xcolor}
\usepackage{enumerate}

\def\BibTeX{{\rm B\kern-.05em{\sc i\kern-.025em b}\kern-.08em
		T\kern-.1667em\lower.7ex\hbox{E}\kern-.125emX}}

\newcommand{\real}{\ensuremath{\mathbb{R}}}

\newcommand{\until}[1]{\{1,\dots, #1\}}
\newcommand{\subscr}[2]{#1_{\textup{#2}}}

\renewcommand{\epsilon}{\varepsilon}

\newcommand{\argmin}{\ensuremath{\operatorname{argmin}}}

\usepackage{psfrag}

\makeatletter
\newcommand\RedeclareMathOperator{%
	\@ifstar{\def\rmo@s{m}\rmo@redeclare}{\def\rmo@s{o}\rmo@redeclare}%
}
\newcommand\rmo@redeclare[2]{%
	\begingroup \escapechar\m@ne\xdef\@gtempa{{\string#1}}\endgroup
	\expandafter\@ifundefined\@gtempa
	{\@latex@error{\noexpand#1undefined}\@ehc}%
	\relax
	\expandafter\rmo@declmathop\rmo@s{#1}{#2}}
\newcommand\rmo@declmathop[3]{%
	\DeclareRobustCommand{#2}{\qopname\newmcodes@#1{#3}}%
}
\@onlypreamble\RedeclareMathOperator
\makeatother

\newcommand{\longthmtitle}[1]{\mbox{}{\bf \textit{(#1).}}}
\theoremstyle{plain}
\newtheorem{thm}{Theorem}

\newtheorem{prop}{Proposition}

\theoremstyle{definition}
\newtheorem{assump}{Assumption}
\newtheorem{cond}{Condition}
\newtheorem{example}{Example}

\theoremstyle{remark}

\newcommand{\Pref}{\ensuremath{\subscr{P}{ref}}}
\DeclareMathOperator{\ones}{\mathbf{1}}
\DeclareMathOperator{\zeros}{\mathbf{0}}
\RedeclareMathOperator{\P}{\mathcal{P}}

\DeclareMathOperator{\E}{\mathcal{E}}
\DeclareMathOperator{\U}{\mathcal{U}}
\DeclareMathOperator{\Ex}{\mathbb{E}}
\DeclareMathOperator{\Px}{\mathbb{P}}
\DeclareMathOperator{\D}{\mathcal{D}}
\DeclareMathOperator{\N}{\mathcal{N}}
\DeclareMathOperator{\G}{\mathcal{G}}
\RedeclareMathOperator{\L}{\mathcal{L}^S}
\DeclareMathOperator{\x}{\mathbf{x}}

\newcommand{\mean}{\operatorname{mean}}
\newcommand{\torhighlight}[1]{\textcolor{black}{#1}}

\begin{document}
	
	\title{
		Distributed Stochastic Nested Optimization via Cubic Regularization
	}
	\author{Tor Anderson \qquad Sonia Mart{\'\i}nez \tnoteref{t1}}
	
	\tnotetext[t1]{Tor Anderson and Sonia Mart{\'\i}nez are with the Department of Mechanical and Aerospace Engineering, University of California, San Diego, CA, USA. Email: {\small {\tt \{tka001, soniamd\}@eng.ucsd.edu}}. This research was supported by the Advanced Research Projects Agency - Energy under the NODES program, Cooperative Agreement DE-AR0000695.} 
	
	\begin{abstract}
		This paper considers a nested stochastic distributed optimization
		problem. In it, approximate solutions to realizations of the
		inner-problem are leveraged to obtain a Distributed Stochastic Cubic
		Regularized Newton (DiSCRN) update to the decision variable of the
		outer problem. We provide an example involving electric vehicle
		users with various preferences which demonstrates that this model is
		appropriate and sufficiently complex for a variety of data-driven
		multi-agent settings, in contrast to non-nested models. The main two
		contributions of the paper are: (i) development of local stopping
		criterion for solving the inner optimization problem which
		guarantees sufficient accuracy for the outer-problem update, and
		(ii) development of the novel DiSCRN algorithm for solving the
		outer-problem and a theoretical justification of its
		efficacy. Simulations demonstrate that this approach is more stable
		and converges faster than standard gradient and Newton outer-problem
		updates in a highly nonconvex scenario, and we also demonstrate that the method extends to an EV charging scenario in which resistive battery losses and a time-of-use pricing model are considered over a time horizon.
		
	\end{abstract}
\maketitle
\section{Introduction}

\textit{Motivation.} 
As applications emerge which are high dimensional and described by
large data sets, the need for powerful optimization tools has never
been greater. In particular, agents in distributed settings are
commonly given a global optimization task where they must sparingly
exchange local information with a small set of neighboring agents for
the sake of privacy and robust scalability. This architecture can,
however, slow down convergence compared to centralized ones, which is
concerning if obtaining the iterative update information is
costly. Gradient-based methods are commonly used due to their
simplicity, but they tend to be vulnerable to slow convergence around
saddle points. Newton-based methods use second-derivative information
to improve convergence, but they are still liable to be slow in areas
where higher order terms dominate the objective function and even
unstable when the Hessian is ill conditioned. A powerful tool for
combating these Newton-based vulnerabilities is imposing a cubic
regularization on the function's second-order Taylor approximation,
but the current work on this technique does not unify
\emph{distributed}, \emph{stochastic}, and \emph{nonconvex}
elements. Motivated by this, we study the adaptation of the Stochastic
Cubic Regularized Newton approach to solve a distributed nested
optimization problem.

\textit{Literature Review.}
One of the most 
widely used 
stochastic optimization method is  
stochastic gradient-based (first-order) methods,
see~\cite{WAG:84,LB:10,LB-FEC-JN:18} as broad references. These
methods are powerful because they necessitate only a small sampling of
the data set 
to compute an update direction at each iterate. However, these
first-order algorithms suffer from slow convergence around
saddle-points~\cite{SD-CJ-JL-MJ-BP-AS:17}, which are
disproportionately more present in higher-dimensional nonconvex
problems~\cite{YD-RP-CG-KC-SG-YB:14}. By contrast, higher-order
Newton-based methods tend to perform more strongly across applications
in terms of number of calls to an oracle or total iterations,
see~\cite{FY-AN-US:16,XW-SM-DG-WL:17} for examples in stochastic
non-strongly convex and nonconvex settings, respectively,
and~\cite{AM-QL-AR:17,TA-CYC-SM:18-auto,RT-HBA-AJ:19} for various multi-agent
examples.

An issue with many of the aforementioned algorithms is they are
vulnerable to slow convergence or instability in the presence of
saddle-points and/or an ill-conditioned Hessian matrix. 
A
growing body of works thus focuses on using a cubic-regularization
term in the second-order Taylor approximation of the objective
function. 
Nesterov and Polyak laid significant groundwork for this method
in~\cite{YN-BTP:06}, and substantial follow-ups are contained
in~\cite{CC-NG-PT:09P1,CC-NG-PT:10P2}, which study adaptive batch
sizes and the effect of inexactness in the cubic submodel on
convergence. Excitement about this topic has grown substantially in
the last few years, with~\cite{YC-JD:19} showing how the global
optimizer of the nonconvex cubic submodel can be obtained under
certain initializations of gradient descent,
and~\cite{NT-MS-CJ-JR-MJ:18} being one of the first thorough analyses
of the algorithm in the traditional stochastic optimization
setting. In~\cite{XC-BJ-TL-SZ:18}, the authors consider the stochastic
setting from an adaptive batch-size perspective and~\cite{CU-AJ:20}
is, to our knowledge, the only existing work in a distributed
application, with an alternative approach that allows for a
communication complexity analysis. 
Both~\cite{XC-BJ-TL-SZ:18} and~\cite{CU-AJ:20} assume convexity,
and~\cite{CU-AJ:20} is nonstochastic. As far as we know, no current
work has unified \emph{distributed}, \emph{stochastic}, and
\emph{nonconvex} elements, particularly in a nested optimization
scenario.

\textit{Statement of Contributions.} We begin the paper by formulating
a nested distributed stochastic optimization problem, where
approximate solutions to realizations of the inner-problem are needed
to obtain iterative updates to the outer problem, and we motivate this
model with an example based on electric vehicle charging
preferences. The contributions of this paper are then twofold. First,
we develop a stopping criterion for a Laplacian-gradient subsolver of
the inner-problem. The stopping criterion can be validated locally by
each agent in the network, and the relationship to solution accuracy
aids the synthesis with the outer-problem update. Second, to that end,
we formulate a distributed optimization model of the stochastic
outer problem and develop a cubic regularization of its second-order
approximation. This formulation lends itself to obtaining a
Distributed Stochastic Cubic-Regularized Newton (DiSCRN) algorithm, and
we provide theoretical justification of its convergence.

\section{Preliminaries}

This section establishes notation\footnote{The set of real numbers,
	real $n$-dimensional vectors, and real $n$-by-$m$ matrices are
	written as $\real, \real^n$, and $\real^{n\times m}$,
	respectively. The transpose of a matrix $A$ is denoted by $A^\top$,
	the $n\times n$ identity matrix is written as $I_n$, and we write
	$\ones_n = (1,\dots,1)^\top\in\real^n$ and $\zeros_n =
	(0,\dots,0)^\top\in\real^n$. Orthogonality of two vectors $x,y\in
	\real^n$ is denoted $x\perp y \leftrightarrow x^\top y = 0$. The
	standard Euclidean norm and the Kronecker product are indicated by
	$\| \cdot \|$, $\otimes$,
	respectively. 
	For a function $f:\real^n\rightarrow \real$, the gradient and
	Hessian of $f$ with respect to $x\in\real^n$ at $x$ are written as
	$\nabla f(x), \nabla^2 f(x)$, respectively. When $f:\real^n\times
	\real^m\rightarrow\real$ takes multiple arguments, we specify the
	differentiation variable(s) as a subscript of $\nabla$. We use $\Ex$ and $\Px$ 
	to denote expectation and probability, $\delta_a$ to denote the Dirac delta function
	centered at $a\in\real$, and $\U[a,b]$ to denote the uniform
	distribution on $[a,b]$.}  and background concepts to be used
throughout the paper.


First, we provide a brief background on the Cubic-Regularized Newton
method. See~\cite{YN-BTP:06} and~\cite{CC-NG-PT:09P1,CC-NG-PT:10P2}
for more information.  Consider the problem of minimizing a (possibly
nonconvex) function $f:\real^d\rightarrow \real$:
\begin{equation}\label{eq:CRN-prob}
\underset{x\in\real^d}{\text{min}} \ f(x).
\end{equation}
One useful iterative model for minimizing $f(x^k)$ when the function
is strictly convex at the current iterate $x^k$ (or, more accurately, if it is strictly convex on some neighborhood of $x^k$) is descent on a second-order Taylor expansion around
$x^k$:
	\begin{equation}\label{eq:so-model}
	\begin{aligned}
	& \quad x^{k+1} =  \underset{x}{\argmin} \ \bigg{\{} f(x^k) + (x-x^k)^\top
	\nabla f(x^k) 
	\\ &+ \frac{1}{2}(x-x^k)^\top \nabla^2 f(x^k) (x-x^k)\bigg{\}} = x^k - \nabla^{-2}f(x^k)\nabla f(x^k).
	\end{aligned}
	\end{equation}
This closed form expression for $x^{k+1}$ breaks down when $f$ is
nonconvex due to some eigenvalues of $\nabla^2 f(x^k)$ having negative
sign. Further, when $\nabla^2 f(x^2)$ is nearly-singular, the update
becomes very large in magnitude and can lead to instability.
For this reason, consider amending the second-order model with a
cubic-regularization term, to
obtain 
the cubic-regularized, third-order model of $f$ at $x^k$ as:
\begin{equation}\label{eq:to-model}
\begin{aligned}
	m_k(x) &\triangleq \bigg{\{}f(x^k) + (x-x^k)^\top \nabla f(x^k) \\ 
	&+ \frac{1}{2}(x-x^k)^\top \nabla^2 f(x^k) (x-x^k) + \frac{\rho}{6}\| x-x^k \|^3\bigg{\}}.
	\end{aligned}
\end{equation}
Here, $\rho$ is commonly taken to be the Lipschitz constant of
$\nabla^2_{xx}f$, which we will formalize in
Section~\ref{sec:prob-form}. 
The update is naturally given by a minimizer to this model:
$x^{k+1} \in \underset{x}{\argmin} \ m_k(x).$
Unfortunately, this model does not beget a closed-form minimizer as
in~\eqref{eq:so-model}, nor is it convex if $f$ is not convex. The
model does, however, become convex for $x$ very far from $x^k$, which
can be seen by computing the Hessian of $m_k$ as $\nabla^2 m_k(x) =
\nabla^2 f(x^k) + \rho \| x-x^k\| I_n$. Additionally, $m_k$ is an
\emph{over-estimator} for $f$, i.e. $m_k (x) \geq f(x), \forall x$. This is seen by considering the cubic
term and recalling Lipschitz properties of $\nabla^2 f$; we describe
this observation in more detail later in the paper.  Therefore, $m_k$
possesses some advantages over other simpler submodels as it possesses
properties of a more standard Newton-based, second-order model while
being sufficiently conservative.

Finally,~\cite{YC-JD:19} recently showed that simply initializing
$x=x^k-r\nabla f(x^k)/\|\nabla f(x^k)\|$ for $r\geq 0$ is sufficient to
show that gradient descent on $m_k$ converges to the global minimizer
of~\eqref{eq:to-model} (under light conditions on $r$ and the gradient
step size).

We refer the reader to~\cite{FB-JC-SM:09} for supplementary notions on Graph
Theory and more background on the Laplacian matrix.

\section{Problem Formulation}\label{sec:prob-form}

This section details the two problem formulations which are of
interest, where the first problem $\P 1$ takes the form of a
stochastic approximation whose cost is a parameterization of the cost
of the second problem $\P 2$. Problem $\P 2$ is a separable
resource allocation problem in which $n$ agents $i\in \N$ must
collectively obtain a solution that satisfies a linear equality
constraint while minimizing the sum of their local costs. (This
problem commonly appears in real-time optimal dispatch for electric
grids with flexible loads and distributed generators, see
e.g.~\cite{CAISO-BPM:18}.) Thus, $\P 1$ can be treated as a nested
optimization, with an objective $F$ that takes stochastic arguments,
and is not necessarily available in closed form if $\P 2$ cannot be
solved directly and/or the distribution $\D$ being unknown. These
problems are stated as
\begin{equation*} 
\P 1: \
\underset{x\in\real^d}{\text{min}} \
F(x) = \Ex_{\chi\sim\D} \left[F_\chi (x)\right].
\end{equation*}
\vspace{-.25cm}
\begin{equation*} 
\begin{aligned}
\P 2: \
\underset{p\in\real^n}{\text{min}} \
f(x,p) &= \sum_{i=1}^n f_i(x,p_i), \\
\text{subject to} \ \sum_{i=1}^n p_i &= \Pref + \hat{\chi} = \Pref + \sum_{i=1}^n \hat{\chi}_i.
\end{aligned}
\end{equation*}
In $\P1$, each $f_i:\real^d \times \real \rightarrow \real$, and $F_\chi(x) \equiv f(x,p^\star)$, where $p^\star$ is the
solution to $\P 2$ for particular realizations $\hat{\chi}_i,$ where $\chi_i\sim\D_i$, i.e. $\chi\sim\D = \D_1 \times \dots \times \D_n$. The elements
$p_i\in\real$ of $p\in\real^n$ and terms $\hat{\chi}_i$ are each associated
with and locally known by agents $i\in\N$, and $\Pref\in\real$ is a
given constant known by a subset of agents (we discuss its
interpretation shortly with an example). First, for $F_\chi$ to be
well defined, it helps if solutions $p^\star$ to $\P 2$ are unique for
fixed $x$ and $\hat{\chi}$, which we now justify with convexity
assumptions for
$f_i$. 
\begin{assump}\longthmtitle{Function Properties: Inner-Problem Argument}\label{assump:fun-inner}
	The local cost functions $f_i$ are twice differentiable and
	$\omega_i$-strongly convex in $p_i$ for any fixed $x$. Further, the
	second derivatives are lower and upper bounded:
	\begin{equation*} 0 < \omega_i\leq \nabla^2_{p_i}f_i(x,p_i) \leq
		\theta_i, \qquad \forall x\in\real^d, p_i\in\real, \text{ and } i\in\N.
	\end{equation*}
	This implies $\forall x\in\real^d, p_i, \hat{p}_i\in\real
	\text{ and } i\in\N$:
	\begin{equation*}\omega_i \|p_i - \hat{p}_i \| \leq \|
	\nabla_{p_i} f_i(x,p_i) - \nabla_{p_i} f_i(x,\hat{p}_i) \|
	\leq \theta_i \| p_i -
	\hat{p}_i\|. 
	\end{equation*}
	We also use the shorthands $\omega \triangleq \min_i{\omega_i}$ and $\theta \triangleq \max_i{\theta_i}$.
\end{assump}
This assumption will be required of our analysis in
Section~\ref{ssec:inner-loop}.
We now state some additional assumptions.
\begin{assump}
	\longthmtitle{Function Properties: Lipschitz Outer-Problem
		Argument}\label{assump:fun-outer-lips} The functions $f_i$ have
	$l_i$-Lipschitz gradients and $\rho_i$-Lipschitz Hessians:
	\begin{equation*}{\small
	\begin{aligned} &\| \nabla f_i(x,p_i) - \nabla f_i(y,p_i) \| \leq l_i \| x - y\|, &&\forall x, y\in\real^d, \forall p_i\in\real, \\ &\| \nabla^2 f_i(x,p_i) - \nabla^2 f_i(y,p_i) \| \leq \rho_i \| x - y\|, &&\forall x, y\in\real^d, \forall p_i\in\real.
	\end{aligned}}
	\end{equation*}
	We also use the shorthands $l\triangleq \max_i l_i$ and $\rho\triangleq \max_i \rho_i$.
\end{assump}
\begin{assump}\longthmtitle{Function Properties: 
		Bounded Variance Outer-Problem Argument}\label{assump:fun-outer-var}
	The function $F_\chi $ possesses the following bounded variance
	properties: 
	\begin{equation*}
		\begin{aligned} &\Ex\left[\|\nabla F_\chi (x) - \nabla F (x)\|^2 \right] \leq \sigma^2_1, \\ 
		&\Ex\left[\|\nabla^2 F_\chi (x) - \nabla^2 F (x)\|^2 \right] \leq \sigma^2_2, \\
		&\|\nabla F_\chi (x) - \nabla F (x)\|^2 \leq M_1 \text{ almost surely}, \\
		&\|\nabla^2 F_\chi (x) - \nabla^2 F (x)\|^2 \leq M_2 \text{ almost surely}.
		\end{aligned}
	\end{equation*}
\end{assump}
\begin{assump}\longthmtitle{Function 
		Properties: Lipschitz Interconnection of Variables}\label{assump:inter-lipsch}
	The gradient and Hessian of the function $f$ with respect to $x$ are
	Lipschitz in $p$; that is, there exists constants $\psi_g, \psi_H > 0$
	such that
	\begin{equation*}
	{\small 
		\begin{aligned}
		&\| \nabla_{x} f(x,p) - \nabla_{x} f(x,\hat{p}) \| \leq \psi_g \| p - \hat{p}\|, \\
		&\| \nabla^2_{xx} f(x,p) - \nabla^2_{xx} f(x,\hat{p}) \| \leq \psi_H \| p - \hat{p}\|, \\ 
		&\qquad \forall x\in\real^d, p, \hat{p}\in\real^n.
		\end{aligned}
	}
	\end{equation*}	
\end{assump}
Assumption~\ref{assump:fun-inner} is relatively common in the convex
optimization literature, and it lends itself to obtaining approximate
solutions to $\P 2$ very quickly with stopping criterion
guarantees. Assumption~\ref{assump:fun-outer-lips} is unanimously
leveraged in literature on Cubic-Regularized Newton methods, as the
constant $\rho$ pertains directly to the cubic submodel, while
Assumption~\ref{assump:fun-outer-var} is a common assumption in the
stochastic optimization literature~\cite{NT-MS-CJ-JR-MJ:18}. We note
that, although Assumptions~\ref{assump:fun-outer-lips}
and~\ref{assump:fun-outer-var} do not give a direct relationship with
the local functions $f_i(x,p_i)$, they do imply an implicit
relationship between $x, p, \text{ and } \D$ in the sense that
solutions $p^\star$ to $\P 2$ (and therefore the distributions $\D_i$)
must be ``well-behaved" in some sense. 
This relationship, along with a broader interpretation of the model
$\P 1$ and $\P 2$, is illustrated more concretely in the following
real-world power distribution example.


\begin{example}\longthmtitle{EV Drivers with PV Generators}\label{example:model}
	Consider two EV drivers who each have an EV charging station and a
	PV generator. 
	The goal of this small grid system is to consume net zero power from
	the perspective of the tie line to the bulk grid, thus $\Pref =
	0$. The distributions $\D_1,\D_2$ represent the power output
	distributions of the PVs, and we consider two scenarios for these in
	this example: (1) a ``sunny day" scenario, where the realizations
	$\chi_1,\chi_2\sim\D_1,\D_2$ of PVs 1 and 2 are deterministic, and
	(2) a ``cloudy day" scenario, where intermittent cloud cover induces
	some uncertainty in the moment-to-moment PV generation.
	
	Let $A\in\{\text{sunny},\text{cloudy}\}$ indicate the weather forecast. The model is then fully described as
	\begin{equation*}
	\D_i = \begin{cases}
	\delta_{1.5}, & A = \text{sunny}, \\
	\U[0,1.5], & A = \text{cloudy}
	\end{cases} \quad \text{for both i=1,2},
	\end{equation*}
	\begin{equation*}
	f_1(x,p_1) = (2x + p_1 - 1)^2, \;
	f_2(x,p_2) = (x + p_2 - 2)^2.
	\end{equation*}
	
	For $x=0$, these quadratic functions\footnote{See~\cite{SB-HF:13} for an example where quadratic costs to EV users are induced by resistive energy losses in the battery model and~\cite{AW-BW-GS:12} for a broad reference on modeling generator dispatch.} have local minima at
	$p^\star = (p_1^\star,p_2^\star ) = (1,2)$, which is
	interpreted as drivers 1 and 2 preferring to charge at rates
	of 1 unit and 2 units, respectively, if there are no external
	incentives. On a sunny day, both PVs 
	deterministically produce $\hat{\chi}_1,\hat{\chi}_2 = 1.5$, which
	effectively balances the unconstrained $p^\star$ and both
	drivers can charge at their preference to maintain $\sum_i p_i
	= \Pref + \sum_i \hat{\chi}_i$.
	
	However, on cloudy days the generation of the PVs is no longer deterministic. Thus, the variable $x$
	comes in to play, which can represent a government credit that
	the drivers value differently. The role of $x$ is to shift the cost
	functions such that the unconstrained minima are near lower
	charging values in consideration of the lower expected
	generation from PVs 1 and 2. The optimal $x^\star$ to $\P 1$
	is the value which gives the lowest expected cost of an
	instance of $\P 2$ given $\hat{\chi}_1,\hat{\chi}_2$ realizations from the
	$A = $ cloudy distributions $\D_1, \D_2$. A more complete model of $\P 2$ could include power flow constraints; in this work, we relax these for simplicity.
	
\end{example}

\section{Distributed Formulation and Algorithm}

In this section, we develop the inner-loop
algorithm used to solve $\P 2$. We then synthesize inexact solutions to $\P 2$ with the DiSCRN algorithm for $\P 1$.

\subsection{Inner Loop Gradient Solver}\label{ssec:inner-loop}

For this section, consider $x$ to be fixed and known by all
agents. Further, let $\hat{\chi}_i$ be fixed (presumably from a realization
of $\D_i$) and known only to agent $i$. We adopt the following
assumption on the initial condition $p^0$.

\begin{assump}\longthmtitle{Feasibility of Inner-Problem Initial Condition}\label{assump:init-cond-p2}
	The agents are endowed with an initial condition which is feasible
	with respect to the constraint of $\P 2$; that is, they each possess
	elements $p_i^0$ of a $p^0$ satisfying $\ones_n^\top p^0 = \Pref + \hat{\chi}.$
\end{assump}

The assumption is easily satisfied in practice by communicating
$\Pref$ to one agent $i$ and setting $p_i^0 = \Pref + \hat{\chi}_i$, with
all other agents $j$ using $p_j^0 = \hat{\chi}_j$. 
An alternative to this assumption consists of reformulating
$\P 2$ with distributed constraints and using a dynamic consensus
algorithm as in~\cite{AC-EM-SHL-JC:18-tac}, which would still retain
exponential convergence. We impose
Assumption~\ref{assump:init-cond-p2} for simplicity. Finally, we
assume connectedness of the communication graph:
\begin{assump}\longthmtitle{Graph Properties}\label{assump:graph-conn}
	The communication graph $\G$ is connected and undirected; that is, a
	path exists between any pair of nodes and, equivalently, its
	Laplacian matrix $L = L^\top \succeq 0$ has rank $n-1$ with
	eigenvalues $0 = \lambda_1 < \lambda_2 \leq \dots \leq \lambda_n$.
\end{assump}

The discretized Laplacian-flow dynamics are given by:
\begin{equation}\label{eq:disc-lap-flow}
p^+ = p - \eta L\nabla_p f(x,p).
\end{equation}
Note that these dynamics are distributed, as the sparsity of $L$
implies each agent need only know $\nabla_{p_i} f_i(x,p_i)$ and
$\nabla_{p_j} f(x,p_j)$ for $j\in\N_i$ to compute $p_i^+$. We now
justify convergence of~\eqref{eq:disc-lap-flow} to the solution
$p^\star$ of $\P 2$:
\begin{prop}\longthmtitle{Convergence of Discretized Laplacian Flow}\label{prop:grad-cvg}
  Let $p^\star \in \real^n$ be the unique minimizer of $\P 2$. Given
  Assumption~\ref{assump:init-cond-p2} on the feasibility of the
  initial condition, Assumption~\ref{assump:graph-conn}, on
  connectivity of the communication graph, and
  Assumption~\ref{assump:fun-inner} on the Lipschitz gradient
  condition of the function gradients, then, under the
  dynamics~\eqref{eq:disc-lap-flow} with $0 < \eta < \frac{2}{\theta
    \lambda_n}$, $p$ converges asymptotically to $p^\star$.
\end{prop}
\begin{proof}
  Using a standard quadratic expansion around the current iterate $p$
  (see e.g. $\S 9.3$ of~\cite{SB-LV:04}) and Lipschitz bounds yields
  $f(p^+) - f(p) \leq \theta\eta^2 / 2 \| L \nabla_p f(x,p)\|^2 \\ -
  \eta \nabla_p f(x,p)^\top L \nabla_p f(x,p)$. Careful treatment of
  the eigenspace of $L$ and some algebraic manipulation shows that
  $f(p^+) - f(p)$ is strictly negative for $\eta$ as in the
  statement. 
\end{proof}

We now provide an additional result on exponential convergence of the state
error with a further-constrained step size as compared to the
statement in Proposition~\ref{prop:grad-cvg}.
\begin{prop}\longthmtitle{Exponential Convergence with Bounded Error}\label{prop:exp-cvg}
	Let Assumptions~\ref{assump:init-cond-p2},~\ref{assump:graph-conn},
	and~\ref{assump:fun-inner} hold as before. For $0 < \eta <
	2\omega\lambda_2/\theta^2\lambda_n^2$, the quantity $\|p - p^\star
	\|$ converges exponentially to zero under the
	dynamics~\eqref{eq:disc-lap-flow}. For $\eta = \omega\lambda_2 / \theta^2 \lambda_n^2$, the rate is $\| p^+ - p^\star \| \leq
	\sqrt{1-\omega^2 \lambda_2^2 / \lambda_n^2 \theta^2}\|p-p^\star\|$,
	and $\|p^K-p^\star\| \leq \Delta$ for $K \geq \log(\Delta/\|p^0 -
	p^\star\|)/\log(\sqrt{1-\omega^2 \lambda_2^2 / \lambda_n^2
		\theta^2})$. 
\end{prop}
\begin{proof}
  Consider $V(p) =
  \|p-p^\star\|^2$. Substituting~\eqref{eq:disc-lap-flow} and applying
  bounds via eigenvalues of $L$, using
  Assumption~\ref{assump:fun-inner}, and $\nu$-strong function
  convexity, we get $V(p^+)\le (\eta^2\lambda_n^2\theta^2 - 2\eta
  \lambda_2 \omega + 1)V(p)$, with $0 < \eta <
  2\lambda_2\omega/\lambda_n^2\theta^2$. The choice of $\eta =
  \omega\lambda_2 / \lambda_n^2 \theta^2$ implies the exponential
  convergence as in the statement. 
\end{proof}

We note that the results of
	Propositions~\ref{prop:grad-cvg} and~\ref{prop:exp-cvg} simply build on a
	Laplacian-projected version of vanilla gradient descent. However, it
	lays some basic groundwork and supplements our main results in the
	next subsection.

With this, we are ready to transition to the discussion on obtaining a DiSCRN update to $\P 1$.

\subsection{Outer-Loop Cubic-Newton Update}

We endow each agent with a local copy $x_i$ of the  variable
$x$, and we let $\x\in\real^{nd}$ be the stacked vector of these local
copies. 
Thus, a distributed reformulation of $\P 1$ is
\begin{equation*}
\begin{aligned}
\overline{\P 1}: \
\underset{\x\in\real^{nd}}{\text{min}} \quad
& \bar{F}(\x) = \Ex_{\chi\sim\D} \left[\bar{F}_{\chi} (\x)\right],  \\
\text{subject to} \quad & (L\otimes I_d)\x = \zeros_{nd},
\end{aligned}
\end{equation*}
where $\bar{F}_{\chi}:\real^{nd}\rightarrow\real$ is analagous to
$F_{\chi}:\real^{d}\rightarrow\real$ in the sense that each agent evaluates $f_i(x_i,p^\star_i)$ with its local copy of
$x_i$. 
Note that the constraint $(L\otimes I_d)\x = \zeros_{nd}$ imposes $x_i
= x_j, \forall i,j$ (Assumption~\ref{assump:graph-conn}), so
$\bar{F}_{\chi}$ and $F_{\chi}$ are equivalent in the agreement
subspace (and $\overline{\P 1}$ is equivalent to $\P 1$). 
Since our problem is nested and stochastic, there is a lack of access to
a closed form expression for $\bar{F}$ and 
$\bar{F}_{\chi}$. Thus, we introduce an
\emph{empirical-risk}, \emph{approximate} objective function.
To this end, let $F^S(\x) = 1/S \sum_{s=1}^S F^\Delta_{\chi^s}(\x)$ 
be the approximation of $\bar{F}$ for $S$ samples of
$\chi^s\sim\D$, where $F^\Delta_{\chi^s}\equiv \sum
f_i(x_i,\tilde{p}^s_i)$ and $\|\tilde{p}^s - p^\star\| \leq \Delta$ for
realization $\chi^s$. In this sense, $F^\Delta_{\chi^s}$ implicitly
depends on $\tilde{p}^s$, and the $\Delta$ superscript is a slight abuse
of notation. For now, the reader can consider $\Delta$ to be a sufficiently small design parameter describing the inexactness of the obtained solutions to $\P 2$; we build on this later. Ultimately, we intend to use batches of $F^S$ rather than the exact
$\overline{F}$ to implement DiSCRN. Consider then the cubic regularized submodel of $F^S$ at some $\x^k$:
	\begin{equation}\label{eq:cub-sub-primal}
	m_S^k (\x) = F^S(\x^k) +
	(\x-\x^k)^\top g^k + \frac{1}{2}
	(\x-\x^k)^\top H^k (\x-\x^k) + \sum_{i=1}^n \frac{\rho_i}{6}\| x_i - x_i^k\|^3,
	\end{equation}
where $g^k = \nabla F^S (\x^k), H^k = \nabla^2 F^S (\x^k)$. Note that there is a slight difference between~\eqref{eq:cub-sub-primal} and
the more standard cubic submodel~\eqref{eq:to-model} in that the
regularization terms are \emph{directly separable}; this is crucial
for a distributed implementation, and our forthcoming analysis justifies that convergence can still be established.
We are interested in finding $\x^+$ which minimizes~\eqref{eq:cub-sub-primal} in the agreement subspace:
	\begin{equation}\label{eq:cub-sub-constr}
	\P 3: \ \underset{\x\in\real^{nd}}{\min} \quad m_S^k(\x), \quad \text{subject to } (L \otimes I_d)\x = \zeros_{nd}.
	\end{equation}
	Therefore, we prescribe the Decentralized Gradient Descent dynamics from~\cite{JZ-WY:18}:
	\begin{equation}\label{eq:saddle-submodel}
	\x^{+,t+1}
	= 
	W\x^{+,t} -\alpha_t \nabla_{\x}m_{S}^k(\x^{+,t}),
	\end{equation}
	where $W = I_{nd} - (1/\lambda_n L\otimes I_d)$ and $\alpha_t \sim 1/t$. Per Proposition 3 and Theorem 2 of~\cite{JZ-WY:18}, $\x^{+,t}$ under the dynamics~\eqref{eq:cub-sub-primal} converges asymptotically to a stationary point of $\P 3$ with $O(1/k)$ convergence in the agreement subspace, i.e. $\| x_i-\overline{x} \|$ approaches zero at a rate $O(1/k)$, where $\overline{x} = \mean{(x_i)}$.

We remark that one could formulate the Lagrangian of $\P 3$ and use a saddle-point method to obtain a useful update $\x^+$. This is more parallel to the work of~\cite{YC-JD:19}, which achieves the global solution via gradient descent in the centralized setting. However, even the existence of a Lagrangian saddle-point is in question when the duality gap is nonzero, so further study is required on that approach.

The above discussion serves to set up the following condition on $\x^{k+1}$:
\begin{cond}\longthmtitle{Subsolver Output}\label{cond:subsolver}
	Let
	$\x^{k+1}$ be the output of a subsolver for $\P 3$. Then,
	\begin{enumerate}[(i)]
		\item $\x^{k+1}$ satisfies $(L\otimes I_d)\x^{k+1} = \zeros_{nd}$.\label{item:cond-i}
		\item For an arbitrarily small constant $c>0$ and some $\epsilon > 0$, $\x^{k+1}$ satisfies $m_S^k (\x^{k+1}) - m_S^k(\x^k) < -c\epsilon\|\x^{k+1} - \x^k \| - c\sqrt{\rho\epsilon} \|\x^{k+1} - \x^k \|^2$. \label{item:cond-ii}
	\end{enumerate}
\end{cond}
Part~(\ref{item:cond-i}) is implied in a linear convergence sense by the result of~\cite{JZ-WY:18} for the subsolver~\eqref{eq:saddle-submodel}. 
The~(\ref{item:cond-ii}) condition is straightforwardly implied by any subsolver that is guaranteed to strictly decrease $m_S^k$, e.g.~\eqref{eq:saddle-submodel}, because $c$ can be taken arbitrarily small. However, it can be seen in the statement of Theorem~\ref{thm:discrn} that small $c$ implies a direct tradeoff with $\Delta$ (becomes small) and/or $S$ (becomes large).

We now give a brief outline of the entire algorithm.
{\begin{center}\underline{DiSCRN Algorithm}\end{center}}
\begin{enumerate}
	\item Initialize $\x^0$ s.t. $(L\otimes I_d)\x^0 = \zeros_{nd}$
	\item Realize $\chi^s$ and initialize $p^0$ per Assumption~\ref{assump:init-cond-p2}\label{item:realize-init}
	\item Implement~\eqref{eq:disc-lap-flow} until $\vert p_i^+ -
	p_i\vert \leq \Delta \eta \lambda_2 \omega /\sqrt{n},
	\forall i$\label{item:stop-crit}
	\item Repeat from step~\ref{item:realize-init} $S$ times, storing $\tilde{p}^s\gets p^+$ at each $s$
	\item Compute locally required elements of $g^k,H^k$
	\item Compute an $\x^{k+1}$ satisfying
	Condition~\ref{cond:subsolver},
	e.g. via~\eqref{eq:saddle-submodel}; repeat from
	step~\ref{item:realize-init}\label{item:outer-loop}
\end{enumerate}
The DiSCRN Algorithm describes a fully distributed algorithm, as each
step can be performed with only local 
information. Ostensibly, $\x^0$ could be initialized arbitrarily, but
the first outer-loop would be a ``garbage" update until agreement is
obtained in step~\ref{item:outer-loop}. Note that Step~\ref{item:stop-crit} relates to a distributed stopping criterion for the subsolver of $\P 2$; this condition produces a solution $p^+$ in finite iterations which is sufficiently close to $p^\star$ for the sake of our analysis. This is detailed more in Theorem~\ref{thm:discrn} and its proof.

\begin{cond}\longthmtitle{Assumptions and Conditions for Theorem~\ref{thm:discrn}}\label{cond:thm1}
	Let $F$ satisfy Assumption~\ref{assump:fun-outer-lips}, on Lipschitz
	gradients and Hessians, and Assumption~\ref{assump:fun-outer-var},
	on variance conditions, and let $f$ satisfy
	Assumption~\ref{assump:inter-lipsch}, on Lipschitz interconnection
	of $x$ and $p$, and Assumption~\ref{assump:fun-inner}, on the Lipschitz condition of the
		function gradients with respect to $p$. Further, let Assumption~\ref{assump:init-cond-p2}, on the feasibility of the
		initial condition for $\P 2$, and Assumption~\ref{assump:graph-conn} on
		connectivity of the communication graph, each hold. Let $\x^{k+1}$ be the output of a subsolver for $\P 3$ that
		satisfies Condition~\ref{cond:subsolver} with $c$, and let $\tilde{p}^s\gets p^+$, where $p^+$ is the returned value under the dynamics~\eqref{eq:disc-lap-flow} satisfying $\vert p_i^+ -
		p_i\vert \leq \Delta \eta \lambda_2 \omega /\sqrt{n},
		\forall i$.
\end{cond}

\begin{thm}\longthmtitle{Convergence of DiSCRN}\label{thm:discrn}
	Let the circumstances of Condition~\ref{cond:thm1} apply here. For $S \geq
	\max\{\frac{M_1}{\bar{c}\epsilon},\frac{\sigma_1^2}{\bar{c}^2\epsilon^2},\frac{M_2}{\bar{c}\sqrt{\rho\epsilon}},\frac{\sigma_2^2}{\bar{c}^2\rho\epsilon}\}O(\log{((\epsilon^{1.5}\zeta\bar{c})^{-1})})$
	with $\bar{c}\epsilon + \psi_g \Delta \leq c \epsilon$ and
	$\bar{c}\sqrt{\rho\epsilon} + \psi_H \Delta \leq
	c\sqrt{\rho\epsilon}$, then, under the DiSCRN algorithm dynamics and for all $\zeta > 0$:
	\begin{enumerate}
		\item $\bar{F}(\x^{k+1}) < \bar{F}(\x^k) $ with
		probability $\geq 1 - \zeta$. \label{item:thm-1}
		\item There is a unique accumulation point $\bar{F}^\star$ of the sequence $\{\bar{F}(\x^k)\}_k$ with probability $1-\zeta$ and, for sufficiently large $k$, the value of $\bar{F}(\x^k)$ is bounded in probability: $\Px(\vert\bar{F}(\x^k) - \bar{F}^\star \vert \geq \kappa) \leq \zeta$ for any $\kappa > 0$.\label{item:thm-2}
		\item If $\bar{F}$ is radially unbounded, then the sequence $\{\x^k\}_k$ converges to a point $\x^\star$ such that $\bar{F}(\x^\star) = \bar{F}^\star$ with probability $1- \zeta$. \label{item:thm-3}
	\end{enumerate} 
\end{thm}
\begin{proof}
	First, we aim to obtain the bound $\| \tilde{p}^s - p^\star\| \leq \Delta$ for each instance $s$ of $\P 2$. The Lipschitz condition of Assumption~\ref{assump:fun-inner} implies
		\begin{equation*}
		{\small
			\begin{aligned}
			\omega \| p - p^\star\| &\leq \| \nabla_p f(x,p) - \nabla_p f(x,p^\star)\|, \\
			\lambda_2 \omega \| p - p^\star\| &\leq \| L(\nabla_p f(x,p) - \nabla_p f(x,p^\star))\| \\
			&= \| L \nabla_p f(x,p) \| = 1/\eta \|p^+ - p\| \leq \Delta \lambda_2 \omega.
			\end{aligned}
		}
		\end{equation*}
		Finally,  $1/\sqrt{n}$  comes from breaking $p^+ - p$
		into components and since, for $v\in\real^n$, if 
		$\vert v_i \vert \leq c/\sqrt{n}$ implies $ \| v \| \leq
		c$.
	
	Turning to $\P 1$, let $g_\star^k = \frac{1}{S} \sum_{s=1}^S \sum_i \nabla_{x_i}
		f_i(x_i^k, p_i^\star)$ and \\ $H_\star^k = \frac{1}{S}
		\sum_{s=1}^S \sum_i \nabla^2_{x_i x_i} f_i(x_i^k, p_i^\star)$.
		Lemma 4 of~\cite{NT-MS-CJ-JR-MJ:18} justifies that for
		arbitrary $\bar{c} > 0$, choosing $S \geq
		\max\{\frac{M_1}{\bar{c}\epsilon},\frac{\sigma_1^2}{\bar{c}^2\epsilon^2},\frac{M_2}{\bar{c}\sqrt{\rho\epsilon}},\frac{\sigma_2^2}{\bar{c}^2\rho\epsilon}\}O(\log{((\epsilon^{1.5}\zeta\bar{c})^{-1})})$
		implies that $\| g_\star^k - \nabla \bar{F}(\x^k) \| \leq
		\bar{c}\epsilon$ and $\|(H_\star^k - \nabla_{\x\x}^2
		\bar{F}(\x^k))v\| \leq
		\bar{c}\epsilon\sqrt{\rho\epsilon}\|v\|, \ \forall v$ with
		probability $1-\zeta$.
	
	Let $\phi^k_g = g^k - g_\star^k, \phi^k_H = H^k - H_\star^k$, where $g^k$ and $H^k$ use the inexact estimates $\tilde{p}^s$ satisfying $\| \tilde{p}^s - p^\star\| \leq \Delta$. Substitutions and applying Assumption~\ref{assump:inter-lipsch} gives:
		\begin{equation*}
		\begin{aligned}
		&\| g^k - \nabla_{\x} \bar{F}(\x^k) \| \leq \| g_\star^k - \nabla_{\x} \bar{F}(\x^k) \| + \| \phi^k_g\| \\
		& \qquad \leq \bar{c}\epsilon + \psi_g \Delta, \\
		&\| (H^k - \nabla^2_{\x\x} \bar{F}(\x^k))v \| \leq \| (H_\star^k - \nabla^2_{\x\x} \bar{F}(\x^k))v \| + \| \phi^k_H\| \\
		& \qquad \leq \bar{c}\sqrt{\rho\epsilon} + \psi_H \Delta, \ \forall v.
		\end{aligned}
		\end{equation*}
		Next, let $\xi^k := \x^{k+1} - \x^k$ for notational convenience. The separable cubic regularized terms of $m_S^k$ can be used to bound the true function value:
		\begin{equation*}
		\begin{aligned}
			&\bar{F}(\x^{k+1}) \leq \bar{F}(\x^{k}) + \nabla \bar{F}(\x^{k})^\top \xi^k + \xi^{k\top} \nabla^2 \bar{F} (\x^k) \xi^k \\
			& \qquad \qquad + \sum_i \rho_i/6 \| x_i^{k+1} - x_i^k \|^3 \Rightarrow \\
			&\bar{F}(\x^{k+1}) - \bar{F}(\x^{k}) \leq m_S^k(\x^{k+1}) - m_S^k(\x^k) \\
			&\quad + (\nabla \bar{F} (x^k) - g^k )^\top \xi^k + 1/2 \xi^{k\top} (\nabla^2 \bar{F}(x^k) - H^k ) \xi^k \\
			& \ \leq m_S^k(\x^{k+1}) - m_S^k(\x^k) \\ & \quad + (\bar{c} \epsilon  + \psi_g \Delta) \| \xi^k \| + (\bar{c} \sqrt{\rho\epsilon} + \psi_H \Delta )\| \xi^k\|^2 \\
			& \ \leq m_S^k(\x^{k+1}) - m_S^k(\x^k) + c\epsilon\|\xi^k\| + c\sqrt{\rho\epsilon}\|\xi^k \|^2 < 0,
		\end{aligned}
		\end{equation*}
		where the first inequality is implied by breaking up
                $\bar{F}_\chi (\x)$ in to its separable local
                functions, applying
                Assumption~\ref{assump:fun-outer-lips}, and noting that
                the inequality carries through the expectation
                operator. Subsequent inequalities are directly
                obtained via substitutions. The lefthand inequality of
                the final line stems from the Theorem statement, and
                the righthand inequality of the final line
                from~(\ref{item:cond-ii}) of Condition~1.
		
		As for statement~\ref{item:thm-2}, the existence of
                the unique accumulation point $\bar{F}^\star$ is a
                consequence of the monotonicity of a sequence of real
                numbers. Bounding the value of this point follows
                directly from the definition of an accumulation point.
		
		Regarding statement~\ref{item:thm-3}, note that, if
                the sequence $\{\x^k\}_k$ is bounded, then the set of
                its accumulation points have a finite norm. This
                follows from radial unboundedness of $\bar{F}$ due to
                $\bar{F}(\x^k)\in \{\bar{F}(\x) \,|\,\bar{F}(\x) \leq
                \bar{F}(\x^0)\}$ being compact (where $\bar{F}(\x^0)$
                is some fixed constant). Then, consider any
                accumulation point $\x_a^\star$ of the sequence
                $\{\x^k\}_k$. Because of the monotonicity of
                $\{\bar{F(\x^k)}\}_k$, it follows directly that
                $\bar{F}(\x_a^\star) = \bar{F}^\star$.
		
\end{proof}


\section{Simulations}

\torhighlight{Our simulation study considers two cases: (i) a synthetic nonconvex case, which demonstrates the superiority of DiSCRN over analagous gradient-based and Newton-based approaches, and (ii) a quadratic convex case, which extends the setting to an electric vehicle charging scenario over a time-horizon in which the price of electricity follows a time-of-use model.}

\subsection{Nonconvex Scenario}

The cost functions $f_i$ for this scenario are represented as:
	\begin{equation*}\label{eq:fns-p}
	\begin{aligned}
	f_i(x,p_i) &= \frac{1}{2}\alpha_i(x) p_i^2 + \beta_i(x) p_i + \gamma_i.
	\end{aligned}
	\end{equation*}
	Each $\alpha_i:\real\rightarrow\real$ is quartic in $x$ and generated according to~\eqref{eq:alpha}, where each $a_i^2$ is determined such that $\min_x \ \alpha_i(x) = \omega_i > 0$ with $\omega_i\in\U[1,5]$ per Assumption~\ref{assump:fun-inner}. The $\beta_i:\real\rightarrow\real$ are (possibly nonconvex) quadratic~\eqref{eq:beta}, and $\gamma_i = 0$.
	\begin{equation}\label{eq:alpha}
	\begin{aligned}
	&\alpha_i(x) = a_i^1(x-z_i^1)(x-z_i^2)(x-z_i^3)(x-z_i^4) + a_i^2, \\
	&a_i^1\in\U[0.5,1.5], z_i^1\in\U[-2,-1], z_i^2\in\U[-1,0], \\ &z_i^3\in\U[0,1], z_i^4\in\U[1,2],
	\end{aligned}
	\end{equation}
	\begin{equation}\label{eq:beta}
	\begin{aligned}
	&\beta_i(x) = b_i^1(x-z_i^5)(x-z_i^6), \\
	&b_i^1\in\U[-1,1], z_i^5\in\U[-2,0],z_i^6\in\U[0,2],
	\end{aligned}
	\end{equation}

We compare our DiSCRN method with
gradient-based and Newton-based updates of the same batch sizes, where
the gradient-like and Newton-like updates are computed via:
	\begin{equation*}
	\begin{aligned}
		m_{g}^k(\x) &= F^S (\x^k) +
		(\x-\x^k)^\top g^k + \sum_i \frac{\eta_g}{2} \|x_i - x_i^k\|^2, \\
		m_{H}^k(\x) = F^S (\x^k) &+ (\x-\x^k)^\top g^k
		+\frac{1}{2} (\x-\x^k)^\top H^k (\x-\x^k)
		+ \sum_i \frac{\eta_H}{2} \|x_i - x_i^k\|^2.
	\end{aligned}
	\end{equation*}
We obtain $\x^{k+1}$ empirically for all three methods by
implementing~\eqref{eq:saddle-submodel} until the updates become very
small. We found that both $\eta_g$ and $\eta_H$ must be sufficiently large
to ensure stability, and
$\nabla^2_x F(x) \succ -\eta_H I_d$ to ensure $m_H^k (x)$ bounded. We take $\Delta = 0.1, S = 20, n = 40, \vert
\E \vert = 120, \Pref = 40, \D_i = \U[0,1.5] \ \forall i, \rho = 50, 
\eta_g = 100, \eta_H = 50$.

We note substantially improved performance of DiSCRN over the more traditional gradient-based and Newton-based approaches. In particular, the trajectory finds a minimizer in roughly half and one-third the number of outer-loop iterations required by Newton and gradient, respectively. It is clear that, for $x^{k+1} \approx x^k$, the cubic regularization is less dominant than the squared regularizations, allowing the DiSCRN trajectory to be influenced more by the problem data $g^k, H^k$. As for the parameters $(\rho, \eta_g, \eta_H)$, $\eta_H = 50$ and $\eta_g = 100$ were roughly the lowest possible values without inducing instability. By contrast, reducing $\rho$ to values $\sim 10^{-1}$ was still stable for DiSCRN. We noticed a clear tradeoff between
	$S$ and $\Delta$, with small $S\sim 10^0$ requiring $\Delta\sim
	10^{-1}$ to converge and large $S\sim 10^3$ converging even for large
	$\Delta\sim 10^2$, which is implied by Theorem~\ref{thm:discrn}. Finally, DiSCRN achieves reduced disagreement compared to gradient and Newton; this could be in part due to~\eqref{eq:saddle-submodel} finding a stationary point of $\P 3$ faster, allotting more iterations where the consensus terms dominate the update.

\begin{figure}
	\centering 
	\includegraphics[scale=0.45]{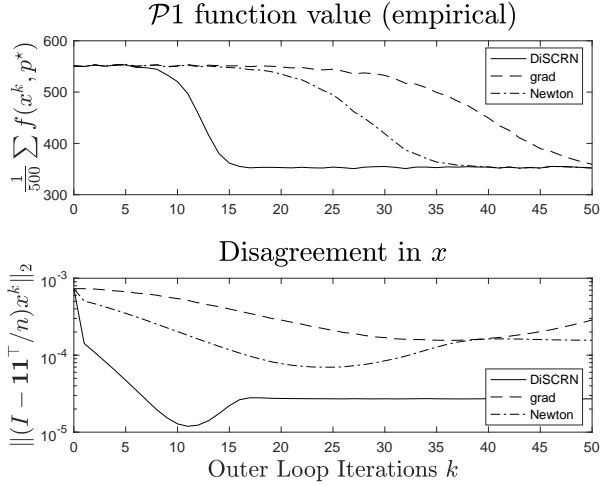}
	\caption{Comparison of CRN method with gradient-based and
		Newton-based approaches. \textbf{Top:} empirical
		approximation of $F(x^k)$, obtained by averaging
		$f(x^k,p^\star)$ over 500 realizations of $\P 2$ at each
		$k$. 
		\textbf{Bottom:} agents' disagreement on the
		value of $x$, quantified by $\| (I-\mathbf{1}\mathbf{1}^\top
		/ n) x^k \|_2$.}\label{fig:sims}
\end{figure}

\subsection{\torhighlight{Electric Vehicle Charging Scenario}}

{\color{black} Consider a system of $n = 25$ electric vehicle users attempting to satisfy a global load constraint while simultaneously minimizing the cost of (i) the actual economic cost of charging (or discharging) at each time $l\in\until{60}$ according to a time-of-use pricing model $P_l$ and (ii) a user-specified preferred charging rate, as in Example~\ref{example:model}. The time-of-use pricing model is characterized by:
\begin{equation*}
P_l = \begin{cases}
2, & l\in\until{20}, \\
4, & l\in\{21,\dots,40\}, \\
1, & l\in\{41,\dots,60\}.
\end{cases}
\end{equation*}
The value $\subscr{P}{ref} = 40$ 
is fixed for all $l$, as
are the distributions $\D_i = \U\left[0.5,1.5\right]$ for all $i$. The
realizations $\chi_{i,l}$ at each time step $l$ represent some net
generation/consumption quantity from each user, e.g. stochastic PV
generation and residential load use.

The following quadratic cost model applies to each EV user at time $l$:
\begin{equation*}
f^l_i(x,p_i) = a_i P_l p_i^2 + b_i P_l p_i + c_i(p_i - \underline{p}_i - d_i x)^2.
\end{equation*}
The total cost to be minimized for $\P 1$ is summed over the entire horizon: $F_\chi (x) \equiv 1/60 \sum_{l=1}^{60} \sum_{i}f^l_i (x,p_i^\star)$.

The first two terms of $f_i^l$ are associated with the real economic
cost of charging/discharging at time $l$, where $a_i,b_i > 0$ are
physical constants associated with the battery model of each user $i$
(e.g. the internal resistance, charge capacity, and open circuit
voltage, see e.g.~\cite{SB-HF:13}). The last term incentivizes charging
close to the user's preferred rate, $\underline{p}_i$, with ``tuning"
coefficients $c_i>0$, which weighs this cost against the economic
cost, and $d_i>0$, which allows for a shift in the preferred charging
rate via an external incentive $x$. 

Users may not only have batteries with different physical constants
$a_i,b_i$, but also different preferences on $\underline{p}_i, c_i,$
and $d_i$. For this study, we simply generate
$a_i,b_i,c_i,d_i\in\U\left[1,3\right]$ and
$\underline{p}_i\in\U\left[0,2\right]$. We choose the simulation
constants $\Delta=0.1, S=20, \vert \E \vert=75, \rho=0.1,
\eta_g=500,\eta_H=1000$. 

The results are plotted in Figure~\ref{fig:sims2}. We again note
superior performance by the DiSCRN algorithm over the gradient-based
and Newton-based methods, where the convergence is much more clearly
linear in this convex case. However, a stronger takeaway of this study
to note that the model was applicable for this extended time-horizon
scenario. A subject of future work is to incorporate battery charging
dynamics and constraints in this scenario which can be gradually
learned by the algorithm, akin to a reinforcement learning setting.

\begin{figure}
	\centering 
	\includegraphics[scale=0.5]{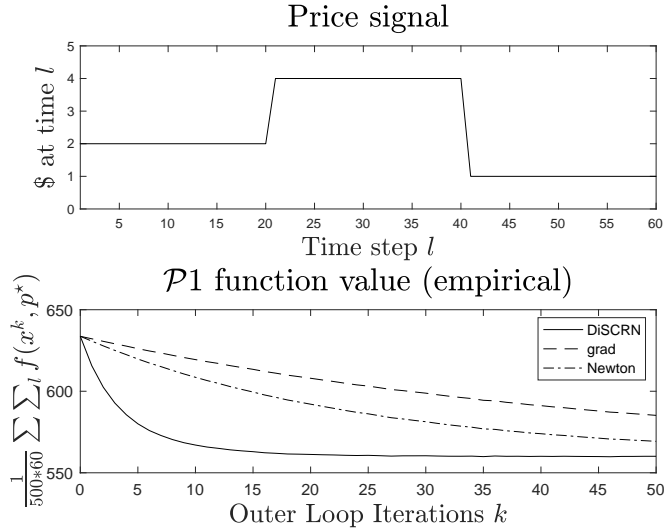}
	\caption{\torhighlight{Comparison of CRN method with gradient-based and
		Newton-based approaches for time-of-use pricing EV charging (convex) scenario. \textbf{Top:} price signal versus time $l$. 
		\textbf{Bottom:} empirical
		approximation of $F(x^k)$, obtained by averaging $1/60 \sum_l f(x^k,p^\star)$ over 500 realizations of $\P 2$ at each $k$.}}\label{fig:sims2}
\end{figure}

}

\section{Conclusion}
Here, we studied a nested, distributed stochastic optimization problem
and applied a Distributed Stochastic Cubic-Regularized Newton (DiSCRN)
algorithm to solve it. In order to compute the DiSCRN update, a batch
of approximate solutions to realizations of the inner-problem are
obtained, and we developed a locally-checkable stopping criterion to
certify sufficient accuracy of these solutions. The accuracy parameter
is directly leveraged in the analysis of the outer-problem, and
simulations justify both faster and more robust convergence properties
than that of comparable gradient-like and Newton-like
approaches. Future work involves developing and analyzing a
saddle-point dynamics approach for solving $\P 3$ (extending the work
of~\cite{YC-JD:19}), extending the analysis to accommodate small
disagreements in the agent states $x_i$, and exploring adaptive batch
size techniques.

\bibliographystyle{abbrv}
\bibliography{alias,SMD-add,JC,SM}

\end{document}